\documentclass[10pt,reqno]{amsart}
  \usepackage{geometry}
\usepackage{amsmath,amssymb,amsthm,mathrsfs,amsfonts,dsfont,hyperref}

\newtheorem{theorem}{Theorem}[section]
\newtheorem{corollary}{Corollary}[section]
\newtheorem{lemma}{Lemma}[section]

\newtheorem{example}{Example}[section]

\title[Filled Julia and Mandelbrot sets]{Filled Julia and Mandelbrot sets for dynamics over the normed real nonassociative algebras}
\author{Jo\~{a}o Carlos da Motta Ferreira$^{*}$\\ and\\ Maria das Gra\c{c}as Bruno Marietto}
\address{Jo\~{a}o Carlos da Motta Ferreira\\
Center for Mathematics, Computation and Cognition\\
Federal University of ABC\\
Santa Ad\'elia Street 166\\
09210-170, Santo Andr\'e, Brazil}
\email{joao.cmferreira@ufabc.edu.br}

\address{Maria das Gra\c{c}as Bruno Marietto\\
Center for Mathematics, Computation and Cognition\\
Federal University of ABC\\
Santa Ad\'elia Street 166\\
09210-170, Santo Andr\'e, Brazil}
\email{graca.marietto@ufabc.edu.br}

\begin{document}
\thispagestyle{empty}
\maketitle

\begin{abstract} In this paper we introduce the notion of dynamical systems over the class of the normed real nonassociative algebras not necessarily finite-dimensional, generalize the classical filled Julia and Mandelbrot sets over the complex numbers, investigate their orbits and make a topological analysis. Applications of the obtained results are made.
\end{abstract}
\let\thefootnote\relax\footnotetext{$^{*}$ Corresponding author 


}

\section{Introduction} 

Filled Julia and Mandelbrot sets arose from the study of dynamical systems over several number systems (see \cite{Blankers,Carleson,Katunin,Yumei}). In this paper we introduce the notion of dynamical systems over the class of the normed real nonassociative algebras not necessarily finite-dimensional, generalize the classical filled Julia and Mandelbrot sets over the complex numbers, investigate their orbits and make a topological analysis of these sets. We apply the obtained results to some classes of finite-dimensional real nonassociative algebras.

Throughout this paper, we assume that all real nonassociative algebras are nonzero and at least $2$-dimensional.

\section{Real nonassociative algebras and normed real nonassociative algebras}

A real linear space $A$ is called a {\it nonassociative algebra} when a multiplication operation $\cdot $ (which is not necessarily associative, commutative or has an identity element) is defined on $A$ satisfying the following properties:
\begin{enumerate}
\item[(i)] $u\cdot (v + w) = u\cdot v + u\cdot w$ and $(u + v)\cdot w = u\cdot w + v\cdot w,$ for all elements $u,v,w$ of $A,$
\item[(ii)] $a(u\cdot v) = (au)\cdot v = u\cdot (av),$ for all real numbers $a$ and elements $u,v$ of $A.$
\end{enumerate}

Let $A$ be a real linear space. A {\it norm} is a mapping $\|\cdot \|$ on $A$ to $\mathds{R}$ satisfying the following: (i) $\|u\|\geq 0,$ for all elements $u$ of $A,$ and $\|u\|=0$ if and only if $u=0,$ (ii) $\|au\|=|a|\|u\|,$ for all real numbers $a$ of $\mathds{R}$ and elements $u$ of $A,$ (iii) $\|u+v\|\leq \|u\|+\|v\|,$ for all elements $u,v$ of $A.$ The pair $(A,\|\cdot \|)$ is called a {\it normed space} if $\|\cdot \|$ is a norm on $A.$

Two norms $\|\cdot \|_{1}$ and $\|\cdot \|_{2}$ on a normed linear space $A$ are called {\it equivalent} if there are positive constants $M_{1}$ and $M_{2}$ such that $\|u\|_{1}\leq M_{2}\|u\|_{2}$ and $\|u\|_{2}\leq M_{1}\|u\|_{1},$ for all elements $u$ of $A.$

A subset $S$ of a normed space $(A,\|\cdot \|)$ is said to be {\it bounded} if there exists a positive real constant $\zeta$ such that $\|u\|\leq \zeta$ for all elements $u$ of $S.$

Let $(A,\|\cdot \|)$ be a normed space. A mapping $s: \mathds{Z}^{+}\rightarrow A$ is called a {\it sequence}, where $\mathds{Z}^{+}$ denotes the set of positive integer numbers. We write $s(n)=s_{n},$ for every integer $n$ of $\mathds{Z}^{+},$ and $s$ is denoted by $(s_{n})_{n\geq 1}.$ A {\it subsequence} of a sequence $(s_{n})_{n\geq 1}$ is a sequence of the form $(t_{l})_{l\geq 1}$ where for each positive integer $l$ there is a positive integer $n_{l}$ such that $\{n_{1}<n_{2}<\cdots <n_{l}<n_{l+1}<\cdots \}$ and $t_{l}=s_{n_{l}}.$ A sequence need not begin with the term of index $1$ but may start with any index. A sequence $s$ whose initial index is a positive integer $n_{0}$ is denoted by $(s_{n})_{n\geq n_{0}}.$ A sequence $(s_{n})_{n\geq 1}$ is said to be {\it convergent} to an element $s_{0}$ of $A$ if for any positive real number $\epsilon $ there exists a positive integer $n_{0}=n_{0}(\epsilon )$ such that $\|s_{n}-s_{0}\|<\epsilon $ when $n\geq n_{0}.$ In this case, we denote this convergence by $\lim _{n\rightarrow +\infty } s_{n}=s_{0}$ and we say that $s_{0}$ is the {\it limit} of $(s_{n})_{n\geq 1}.$

A subset $C$ of a normed space $(A,\|\cdot \|)$ is said to be {\it closed} if the limit of any convergent sequence of points of $C$ belongs to $C.$

An {\it open ball} with center at a point $u$ of $A$ and having radius $\rho >0$ is defined to be the set $\Delta (u,\rho )=\{v\,|\,v \,\textrm{is any element of}\, A\,\textrm{satisfying}\, \|v-u\|<\rho \}.$

Let $(A,\|\cdot \|)$ be a normed space such that $A$ is a real nonassociative algebra. We say that the norm $\|\cdot \|$ is {\it submultiplicative} if the inequality $\|u\cdot v\|\leq \|u\|\|v\|$ holds for all elements $u,v$ of $A.$ A {\it normed real nonassociative algebra} is a normed space $(A,\|\cdot \|)$ such that the norm is submultiplicative. We say that the norm $\|\cdot \|$ is {\it weakly submultiplicative} if there is a positive real number $M_{A}=(A,\|\cdot \|)$ such that $\|u\cdot v\|\leq M_{A}\|u\|\|v\|,$ for all elements $u,v$ of $A.$ In this case, $M_{A}\|\cdot \|$ is a submultiplicative norm on $A$ equivalent to norm $\|\cdot \|$ and so converts $A$ into a normed real nonassociative algebra. We say that the norm $\|\cdot \|$ satisfies the {\it square property} if $\|u^{2}\|=\|u\|^{2},$ for all elements of $A,$ and that satisfies the {\it square inequality} if there is a positive real number $\eta =\eta (A,\|\cdot \|)$ such that $\|u\cdot u\|\geq \eta \|u\|^{2},$ for all elements $u$ of $A.$ It is straightforward that, if a norm $\|\cdot \|$ satisfies the square inequality, then so does every equivalent norm to $\|\cdot \|,$ on $A.$ 

For more details on the theory of nonassociative normed algebras see \cite{Cabrera1,Cabrera2}.

{\it Furthermore, the product of two elements is represented by writing the elements in juxtaposition. Thus $uv$ means the product $u\cdot v,$ of any two elements $u,v$ of $A.$}

\section{Filled Julia and Mandelbrot sets}

In this section, we generalized the notions of the filled Julia and Mandelbrot sets over the class of the normed real nonassociative algebras based on the classical notions of filled Julia and Mandelbrot sets defined over several numbers systems.

Let $S$ be a subset of a normed real nonassociative algebra $(A,\|\cdot \|).$ A {\it dynamical system} is any mapping $f:S\rightarrow S.$ If $f:A\rightarrow A$ is a dynamical system and $u$ is an arbitrary element of $A,$ we write $f^{1}(u)=f(u)$ and $f^{n}(u)=f(f^{n-1}(u))$ for every positive integer $n.$ The set $\{f^{n}(u)\, |\, n\, \textrm{is a positive integer}\}$ is called the {\it orbit} of $u.$ We say that an element $u$ of $A$ has {\it bounded orbit} if there is a positive real number $\zeta $ such that $\|f^{n}(u)\|\leq \zeta ,$ for every positive integer $n.$ When the orbit of an element $u$ is not bounded we say that it is {\it unbounded}.

For any element $c$ of $A$ consider the mapping $f_{c}(u) = u^{2}+c,$ for all elements $u$ of $A.$

The {\it filled Julia set} associated with $f_{c},$ denoted by $\mathcal{K}_{A}(f_{c}),$ is the set of elements $u$ of $A$ for which there exists a positive real number $\zeta _{u}$ such that the inequality
{\allowdisplaybreaks\begin{eqnarray*}\allowdisplaybreaks
\|f^{n}_{c}(u)\|\leq \zeta _{u}
\end{eqnarray*}}
is satisfied, for every positive integer $n.$

The {\it Mandelbrot set}, denoted by $\mathcal{M}_{A},$ is the set of elements $c\in A$ for which there exists a positive real number $\zeta _{c}$ such that the inequality
{\allowdisplaybreaks\begin{eqnarray*}\allowdisplaybreaks
\|f^{n}_{c}(0)\|\leq \zeta _{c}
\end{eqnarray*}}
is satisfied, for every positive integer $n.$

\subsection{Orbit analysis }

\begin{theorem}\label{thm31} Let $A$ be a real nonassociative algebra such that $(A,\|\cdot \|)$ is a normed space whose norm satisfies the square inequality. For every positive real number $\eta $ satisfying $\|u^{2}\|\geq \eta \|u\|^{2},$ for all elements $u$ of $A,$ holds the following: for every element $c$ of $A,$ if for some element $u$ of $A$ and a positive integer $n_{0}$ holds
{\allowdisplaybreaks\begin{eqnarray*}\allowdisplaybreaks
\|f^{n_{0}}_{c}(u)\|> \lambda=\max \{2\eta ^{-1},\|c\|\},
\end{eqnarray*}}
then the orbit of the element $u$ is unbounded.
\end{theorem}
\begin{proof} Fix a positive real $\eta ,$ according to the hypothesis on the norm $\|\cdot \|.$ Let $c$ be an arbitrary element of $A$ and consider an element $u$ satisfying $\|u\|>\lambda .$ Then $\|u\|>\|c\|$ and
{\allowdisplaybreaks\begin{eqnarray*}\allowdisplaybreaks
\|f_{c}(u)\|&=&\|u^{2}+c\|\geq \|u^{2}\|-\|c\|\\
&>&  \|u^{2}\|-\|u\|\geq \eta \|u\|^{2}-\|u\|=\|u\|(\eta \|u\|-1).
\end{eqnarray*}}
Write $\displaystyle \eta \|u\|-1=1+\delta ,$ where $\delta $ is a positive real number. Then
{\allowdisplaybreaks\begin{eqnarray*}\allowdisplaybreaks
\|f_{c}(u)\|>\|u\|(1+\delta )>\|u\|>\lambda .
\end{eqnarray*}}
Assume that for some positive integer $l$ holds $\|f^{l}_{c}(u)\|>\|u\|(1+\delta )^{l}.$ Then for the integer $l+1$ we have
{\allowdisplaybreaks\begin{eqnarray*}\allowdisplaybreaks
\|f^{l+1}_{c}(u)\|&=&\|f_{c}(f^{l}_{c}(u))\|=\|\big(f^{l}_{c}(u)\big)^{2}+c\|\\
&\geq &\|\big(f^{l}_{c}(u)\big)^{2}\|-\|c\|> \eta \|f^{l}_{c}(u)\|^{2}-\|f^{l}_{c}(u)\|\\
&=& \|f^{l}_{c}(u)\|\big( \eta \|f^{l}_{c}(u)\|-1\big)> \|u\|(1+\delta )^{l}\big( \eta \|u\|-1\big)\\
&=&\|u\|(1+\delta )^{l+1}.
\end{eqnarray*}}
Thus, the principle of mathematical induction allows us to conclude that\break $\lim _{n\rightarrow +\infty} \|f^{n}_{c}(u)\|=+\infty .$ One gets the desired result by taking $f^{n_{0}}_{c}(u)$ in the place of $u.$
\end{proof}

\begin{corollary}\label{co31} Let $A$ be a real nonassociative algebra such that $(A,\|\cdot \|)$ is a normed space whose norm satisfies the square inequality. For every positive real number $\eta $ satisfying $\|u^{2}\|\geq \eta \|u\|^{2},$ for all elements $u$ of $A,$  holds the following: for every element $c$ of $A,$ if for some element $u$ of $A$ holds $\|u\|>\lambda=\max \{2\eta ^{-1},\|c\|\},$ then the orbit of the element $u$ is unbounded.
\end{corollary}

\begin{theorem}\label{thm32} Let $A$ be a real nonassociative algebra such that $(A,\|\cdot \|)$ is a normed space whose norm satisfies the square inequality. For every positive real number $\eta $ satisfying $\|u^{2}\|\geq \eta \|u\|^{2},$ for all elements $u$ of $A,$  holds the following: for every element $c$ of $A,$ if for some positive integer $n_{0}$ holds
{\allowdisplaybreaks\begin{eqnarray*}\allowdisplaybreaks
\|f^{n_{0}}_{c}(0)\|> 2\eta ^{-1},
\end{eqnarray*}}
then the orbit of the element $0$ is unbounded.
\end{theorem}
\begin{proof} Fix a positive real $\eta ,$ according to the hypothesis on the norm $\|\cdot \|.$ Let $c$ be an arbitrary element of $A$ satisfying $\|c\|>2\eta ^{-1}.$ Then $\|f_{c}(0)\|=\|c\|>2\eta ^{-1}$ and
{\allowdisplaybreaks\begin{eqnarray*}\allowdisplaybreaks
\|f^{2}_{c}(0)\|&=&\|f_{c}(f_{c}(0))\|=\|c^{2}+c\|\geq \|c^{2}\|-\|c\|\\
&\geq &\eta \|c\|^{2}-\|c\|=\|c\|(\eta \|c\|-1).
\end{eqnarray*}}
Write $\eta \|c\|-1=1+\delta ,$ where $\delta $ is a positive real number. Then
{\allowdisplaybreaks\begin{eqnarray*}\allowdisplaybreaks
\|f^{2}_{c}(0)\|>\|c\|(1+\delta )>\|c\|>2\eta ^{-1}.
\end{eqnarray*}}
Assume that for some positive integer $l$ holds $\|f^{l}_{c}(0)\|>\|c\|(1+\delta )^{l-1}.$ Then for the integer $l+1$ we have
{\allowdisplaybreaks\begin{eqnarray*}\allowdisplaybreaks
\|f^{l+1}_{c}(0)\|&=&\|f_{c}(f^{l}_{c}(0))\|=\|\big(f^{l}_{c}(0)\big)^{2}+c\|\\
&\geq &\|\big(f^{l}_{c}(0)\big)^{2}\|-\|c\|> \eta \|f^{l}_{c}(0)\|^{2}-\|f^{l}_{c}(0)\|\\
&=& \|f^{l}_{c}(0)\|\big( \eta \|f^{l}_{c}(0)\|-1\big)> \|c\|(1+\delta )^{l-1}\big( \eta \|c\|-1\big)\\
&=&\|c\|(1+\delta )^{l}.
\end{eqnarray*}}
Thus, the principle of mathematical induction allows us to conclude that\break  $\lim _{n\rightarrow +\infty} \|f^{n}_{c}(0)\|=+\infty .$ One gets the desired result by taking $f^{n_{0}}_{c}(0)$ in the place of $c.$
\end{proof}

\begin{corollary}\label{co32} Let $A$ be a real nonassociative algebra such that $(A,\|\cdot \|)$ is a normed space whose norm satisfies the square inequality. For every positive real number $\eta $ satisfying $\|u^{2}\|\geq \eta \|u\|^{2},$ for all elements $u$ of $A,$ holds the following: if for an element $c$ of $A$ holds $\|c\|>2\eta ^{-1},$ then the orbit of the element $0$ is unbounded.
\end{corollary}

\subsection{Topological analysis}

The theorem that follows is a generalization of the classical results on the compactness of filled Julia sets \cite[Proposition 14.2]{Falconer} and Mandelbrot sets \cite[Theorem, pp. 338]{Gamelin}, originally proved on the complex dynamical systems, for the general case of the dynamics over the normed real nonassociative algebras not necessarily finite-dimensional.

\begin{theorem}\label{thm33} Let $A$ be a real nonassociative algebra such that $(A,\|\cdot \|)$ is a normed space whose norm satisfies the square inequality. For every positive real number $\eta $ satisfying $\|u^{2}\|\geq \eta \|u\|^{2},$ for all elements $u$ of $A,$ hold the following:
\begin{enumerate}
\item[(i)] for every element $c$ of $A$ the filled Julia set $\mathcal{K}_{A}(f_{c})$ associated with $f_{c}$ is a bounded and closed subset of $A;$
\item[(ii)] the Mandelbrot set $\mathcal{M}_{A}$ is a bounded and closed subset of $A.$
\end{enumerate}
\end{theorem}
\begin{proof} (i) Let $c$ be an element of $A$ and $(u_{l})_{l\geq 1}$ be a sequence of elements of $\mathcal{K}_{A}(f_{c})$ that converges to a limit $u$ of $A.$ First, we observe that if $\lambda =\max \{2\eta ^{-1},\|c\|\},$ then $\|u_{l}\|\leq \lambda ,$ $\|u\|\leq \lambda $ and $\|f^{n}_{c}(u_{l})\|\leq \lambda ,$ for all positive integers $l$ and $n,$ by the hypotheses on the sequence $(u_{l})_{l\geq 1},$ the limit $u$ and the Theorem \ref{thm31}. Since $(f^{n}_{c})_{n\geq 1}$ is a sequence of continuous mappings on $A,$ then $\lim _{l\rightarrow +\infty } f^{n}_{c} (u_{l})=f^{n}_{c}(u)$ and 
{\allowdisplaybreaks\begin{eqnarray*}\allowdisplaybreaks
\|f^{n}_{c}(u)\|=\lim _{l\rightarrow +\infty } \|f^{n}_{c} (u_{l})\|\leq \lambda ,
\end{eqnarray*}}
for every positive integer $n.$ It follows that the orbit of $u$ is bounded which shows that $u$ belongs to $\mathcal{K}_{A}(f_{c}).$ Therefore, the filled Julia set $\mathcal{K}_{A}(f_{c})$ is closed. The boundedness of $\mathcal{K}_{A}(f_{c})$ follows as an immediate consequence of the Corollary \ref{thm31}.\\\\
\noindent (ii) Let $(c_{l})_{l\geq 1}$ be a sequence of elements of $\mathcal{M}_{A}$ that converges to a limit $c$ of $A.$ It follows that $\|c_{l}\|\leq 2\eta ^{-1},$ $\|c\|\leq 2\eta ^{-1}$ and $\|f^{n}_{c_{l}}(0)\|\leq 2\eta ^{-1},$ for all positive integers $l$ and $n,$  by Theorem \ref{thm32}. Since there is only one algebraic way to express $f^{n}_{c_{l}}(0),$ in the variable $c_{l},$ for every positive integer $l,$ then $(f^{n}_{c_{l}}(0))_{l\geq 1}$ is a convergent sequence in $A$ with $\lim _{l\rightarrow +\infty } f^{n}_{c_{l}}(0)=f^{n}_{c}(0)$ and 
{\allowdisplaybreaks\begin{eqnarray*}\allowdisplaybreaks
\|f^{n}_{c}(0)\|=\lim _{l\rightarrow +\infty } \|f^{n}_{c_{l}}(0)\|\leq 
\lambda ,
\end{eqnarray*}}
for all positive integer $n.$ It follows that the orbit of $0$ is bounded which shows that $c$ belongs to $\mathcal{M}_{A}.$ Therefore, the Mandelbrot set $\mathcal{M}_{A}$ is closed. The boundedness of $\mathcal{M}_{A}$ follows as an immediate consequence of the Corollary \ref{co32}.
\end{proof}

\section{$m$-dimensional real nonassociative algebras, multiplication tables and norms}

A nonzero element $u$ of a $m$-dimensional real nonassociative algebra $\mathds{R}^{m}$ is called {\it idempotent} if it satisfies the condition $u^{2}=u.$

For a subset $S$ of a real nonassociative algebra $\mathds{R}^{m}$ the smallest subalgebra of $\mathds{R}^{m}$ containing $S$ is called {\it subalgebra generated} by $S.$ For a subalgebra $S$ of $\mathds{R}^{m}$ we define the following chains of subsets:
{\allowdisplaybreaks\begin{eqnarray*}\allowdisplaybreaks
S^{(0)}=S \,\,\textrm{and}\,\, S^{(k)}=S^{(k-1)}\cdot S^{(k-1)}
\end{eqnarray*}}
The subset $S^{(k)}$ is called the {\it $k$-th solvable power} of $S.$ An subalgebra $S$ of $\mathds{R}^{m}$ is called {\it solvable} if $S^{(k)}=0$ for some positive integer $k.$ The smallest integer $k$ with this property is called the {\it index of solvability} of $S.$

The multiplicative structure of a real nonassociative algebra $\mathds{R}^{m}$ is completely determined by its multiplication table for any basis. The standard presentation of a multiplication table is given by
{\allowdisplaybreaks\begin{eqnarray}\allowdisplaybreaks\label{tbmg}
\begin{tabular}{c|cccccc}
   $\cdot $ & $e_{1}$ & $\cdots $ & $e_{j}$ & $\cdots $ & $e_{m}$ \\ \hline
    $e_{1}$ &         &           &         &        &   \\
  $\vdots $ &         &           & $\vdots $ &        &   \\
    $e_{i}$ &         & $\cdots $  &$\sum _{k=1}^{m} \alpha _{ijk}e_{k}$  & $\cdots $ &   \\
  $\vdots $ &         &           &   $\vdots $      &  &   \\
    $e_{m}$ &         &           &         &  &   \\
\end{tabular}
\end{eqnarray}}
where $\{e_{1},e_{2},\cdots ,e_{m}\}$ is a basis and $\alpha _{ijk}$ $(i,j,k=1,\cdots ,m)$ are real numbers.

\subsection{Norm on $\mathds{R}^{m}$}

In this section we build a norm on an arbitrary real nonassociative  algebra $\mathds{R}^{m},$ whose original idea was developed in \cite[Theorem 5.1.]{Cook1}, for the case of an associative algebra, and has been refined and improved in \cite[Theorem 2.2.]{Cook2}, for the case of a product algebra with weighted $p$-norm.

Given a real nonassociative algebra $\mathds{R}^{m}$ with a basis $\mathcal{B}=\{e_{1},e_{2},\cdots ,e_{m}\}$ and a multiplication table given by (\ref{tbmg}) we may define, in the underlying linear space, an inner-product on $\mathds{R}^{m}$ by a bilinearly extending $g:\mathds{R}^{m}\times \mathds{R}^{m} \rightarrow \mathds{R},$ where $g(e_{i},e_{j})=\delta _{ij}$ $(i,j=1,2,\cdots ,m)$ and $\delta _{ij}$ is the Kronecker delta function. The {\it $g$-induced norm}
{\allowdisplaybreaks\begin{eqnarray}\allowdisplaybreaks\label{gnorm}
\|u\|^{2}=g(u,u),
\end{eqnarray}}
for all elements $u$ of $\mathds{R}^{m},$ has $\|e_{i}\|=1$ $(i=1,2,\cdots ,m),$ and satisfies 
{\allowdisplaybreaks\begin{eqnarray*}\allowdisplaybreaks
\|u\|^{2}=\sum _{i=1}^{m} a_{i}^{2},
\end{eqnarray*}}
for all elements $u=\sum _{i=1}^{m} a_{i}e_{i},$ where $a_{i}$ $(i=1,2,\cdots ,m)$ are real numbers.

\begin{lemma}\label{l41} The norm (\ref{gnorm}) is weakly submultiplicative with 
{\allowdisplaybreaks\begin{eqnarray*}\allowdisplaybreaks
M_{\mathds{R}^{m}}=\sqrt{\sum _{k=1}^{m} \big(\sum _{i=1}^{m} \sum _{j=1}^{m} \alpha _{ijk}^{2}\big)}.
\end{eqnarray*}}
\end{lemma}
\begin{proof} For any elements $u=\sum _{i=1}^{m} a_{i}e_{i}$ and $v=\sum _{i=1}^{m} b_{i}e_{i}$ where $a_{i},b_{i}$ $(i=1,\cdots ,m)$ are real numbers,  we have 
{\allowdisplaybreaks\begin{eqnarray*}\allowdisplaybreaks
\|uv\|^{2}
&=&\|(\sum _{i=1}^{m} a_{i}e_{i}) (\sum _{i=1}^{m} b_{i}e_{i})\|^{2}=\|\sum _{i=1}^{m} \sum _{j=1}^{m} a_{i}b_{j}e_{i}e_{j}\|^{2}\\
&=&\|\sum _{i=1}^{m} \sum _{j=1}^{m} a_{i}b_{j}\big( \sum _{k=1}^{m} \alpha _{ijk} e_{k}\big)\|^{2}=\|\sum _{k=1}^{m} \big(\sum _{i=1}^{m} \sum _{j=1}^{m} \alpha _{ijk}a_{i}b_{j}\big) e_{k}\|^{2}\\
&=&\sum _{k=1}^{m} \big(\sum _{i=1}^{m} \sum _{j=1}^{m} \alpha _{ijk}a_{i}b_{j}\big)^{2}\leq \sum _{k=1}^{m} \big(\sum _{i=1}^{m} \sum _{j=1}^{m} |\alpha _{ijk}||a_{i}||b_{j}|\big)^{2}\\
&\leq &\sum _{k=1}^{m} \big(\sum _{i=1}^{m} \sum _{j=1}^{m} \alpha _{ijk}^{2}\big)\big(\sum _{i=1}^{m} \sum _{j=1}^{m} a_{i}^{2}b_{j}^{2}\big)\,\textrm{(Cauchy–Schwarz inequality)}\\
&=&\big(\sum _{k=1}^{m} \big(\sum _{i=1}^{m} \sum _{j=1}^{m} \alpha _{ijk}^{2}\big)\big)\big(\sum _{i=1}^{m} \sum _{j=1}^{m} a_{i}^{2}b_{j}^{2}\big)\\
&=&\big(\sum _{k=1}^{m} \big(\sum _{i=1}^{m} \sum _{j=1}^{m} \alpha _{ijk}^{2}\big)\big)\big(\sum _{i=1}^{m} a_{i}^{2}\big)\big(\sum _{j=1}^{m}b_{j}^{2}\big)\\
&=&\big(\sum _{k=1}^{m} \big(\sum _{i=1}^{m} \sum _{j=1}^{m} \alpha _{ijk}^{2}\big)\big)\|u\|^{2}\|v\|^{2}
\end{eqnarray*}}
which yields $\|uv\|\leq \Big(\sqrt{\sum _{k=1}^{m} \big(\sum _{i=1}^{m} \sum _{j=1}^{m} \alpha _{ijk}^{2}\big)}\Big)\|u\|\|v\|.$ 
\end{proof}

Given that any two norms on the linear space $\mathds{R}^{m}$ are equivalent and if a norm $\|\cdot \|$ satisfies the square inequality then so does every equivalent norm to $\|\cdot \|,$ in the rest of this paper all the norms considered in any real nonassociative algebras $\mathds{R}^{m}$ will be type (\ref{gnorm}).

\begin{theorem}\label{thm41} Let $CD(n,\mathds{R})$ be the real nonassociative algebra $\mathds{R}^{2^{m}}$ with a basis $\mathcal{B}=\{e_{0}, \cdots ,e_{2^{m}-1}\}$ and a multiplication table satisfying the properties (1), (2) and (3), as given in \cite[pp. 56]{Cohen}. Then, the norm (\ref{gnorm}) satisfies the square property.
\end{theorem}
\begin{proof} Let $u=\sum _{k=1}^{m} a_{k}e_{k}$ be any element of $CD(n,\mathds{R}).$ Then
{\allowdisplaybreaks\begin{eqnarray*}\allowdisplaybreaks
\|u^{2}\|^{2}&=&\|(\sum _{k=0}^{2^{m}-1} a_{k}e_{k}) (\sum _{k=0}^{2^{m}-1} a_{k}e_{k})\|^{2}\\
&=&\|(a_{0}^{2}-\sum _{k=1}^{2^{m}-1} a_{k}^{2})e_{0} +\sum _{k=1}^{2^{m}-1} (2a_{0}a_{k}) e_{k}\|^{2}\\
&=&(a_{0}^{2}-\sum _{k=1}^{2^{m}-1} a_{k}^{2})^{2}+\sum _{k=1}^{2^{m}-1} (2a_{0}a_{k})^{2}\\
&=&(a_{0}^{2})^{2}+(\sum _{k=1}^{2^{m}-1} a_{k}^{2})^{2}-2a_{0}^{2}(\sum _{k=1}^{2^{m}-1} a_{k}^{2})+\sum _{k=1}^{2^{m}-1} (2a_{0}a_{k})^{2}\\
&=&(a_{0}^{2})^{2}+(\sum _{k=1}^{2^{m}-1} a_{k}^{2})^{2}+2a_{0}^{2}(\sum _{k=1}^{2^{m}-1} a_{k}^{2})\\
&=&(a_{0}^{2}+\sum _{k=1}^{2^{m}-1} a_{k}^{2})^{2}\\
&=&(\|u\|^{2})^{2}.
\end{eqnarray*}}
This shows that $\|u^{2}\|=\|u\|^{2},$ for all elements of $CD(n,\mathds{R}).$
\end{proof}

As a consequence of Theorems \ref{thm33} and \ref{thm41}, we have the following result.

\begin{corollary}\label{} Let $CD(n,\mathds{R})$ be the real nonassociative algebra $\mathds{R}^{2^{m}}$ with a basis $\mathcal{B}=\{e_{0}, \cdots ,e_{2^{m}-1}\}$ and a multiplication table satisfying the properties (1), (2) and (3), as given in \cite[pp. 56]{Cohen}. Then, 
\begin{enumerate}
\item[(i)] for every element $c$ of $CD(n,\mathds{R})$ the filled Julia set $\mathcal{K}_{CD(n,\mathds{R})}(f_{c})$ associated with $f_{c}$ is a compact set of $CD(n,\mathds{R});$
\item[(ii)] the Mandelbrot set $\mathcal{M}_{CD(n,\mathds{R})}$ is a compact set of $CD(n,\mathds{R}).$
\end{enumerate}
\end{corollary}

\section{$2$-dimensional real nonassociative algebras}

The best known two-dimensional real nonassociative algebras are the number systems of the form $a+bi$ with addition rule
{\allowdisplaybreaks\begin{eqnarray*}\allowdisplaybreaks
(a+bi)+(c+di)=(a+c)+(b+d)i
\end{eqnarray*}}
and multiplication rule
{\allowdisplaybreaks\begin{eqnarray*}\allowdisplaybreaks
(a+bi)\cdot (c+di)=ac+adi+bci+bdi^{2}=(ac+bdp)+(ad+bc+bdq)i,
\end{eqnarray*}}
where $i^{2}=p+qi,$ with $p$ and $q$ two fixed real numbers, which can be reduced up to isomorphisms to one of the following three types: (i) the number system with $i^{2}=-1$ (called complex number algebra $\mathds{C}$); (ii) the number system with $i^{2}=1$ (called perplex or hyperbolic number algebra $\mathds{P}$); (iii) the number system with $i^{2}=0$ (called dual number algebra $\mathds{D}$). They all have the set $\{1,i\}$ as a basis, where $1$ is the unit element \cite{Kantor}.

In the case of an arbitrary two-dimensional real nonassociative algebra the multiplication table takes the form
{\allowdisplaybreaks\begin{eqnarray*}\allowdisplaybreaks
\begin{tabular}{c|cc}
\multicolumn{3}{c}{Table I} \\
   $\cdot $ & $e_{1}$ & $e_{2}$ \\ \hline
  $e_{1}$ & $a_{11}e_{1}+b_{11}e_{2}$ & $a_{12}e_{1}+b_{12}e_{2}$ \\
  $e_{2}$ & $a_{21}e_{1}+b_{21}e_{2}$ & $a_{22}e_{1}+b_{22}e_{2}$
\end{tabular},
\end{eqnarray*}}
where $\{e_{1},e_{2}\}$ is a basis and $a_{ij},b_{ij}$ are real numbers.

We denote $A=a_{12}+a_{21}$ and $B=b_{12}+b_{21}.$

We are interested only in algebras of dimension two with at least one idempotent element since they are closer to the complex, perplex and dual numbers. An arbitrary two-dimensional real nonassociative algebra $\mathds{R}^{2}$  with at least one idempotent admits a multiplication table as follows
{\allowdisplaybreaks\begin{eqnarray*}\allowdisplaybreaks
\begin{tabular}{c|cc}
\multicolumn{3}{c}{Table II} \\
   $\cdot $ & $e_{1}$ & $e_{2}$ \\ \hline
  $e_{1}$ & $e_{1}$ & $a_{12}e_{1}+b_{12}e_{2}$ \\
  $e_{2}$ & $a_{21}e_{1}+b_{21}e_{2}$ & $a_{22}e_{1}+b_{22}e_{2}$
\end{tabular},
\end{eqnarray*}}
with respect to suitable choice of a basis $\{e_{1},e_{2}\}$ and real numbers $a_{ij},b_{ij}.$

\begin{lemma}\label{l51} Let $\mathds{R}^{2}$ be a real nonassociative algebra with an idempotent $u$ and a nonzero element $v.$ Then the subalgebra generated by $v$ is solvable if and only if $v^{2}=0.$
\end{lemma}
\begin{proof} If $v^{2}\neq 0,$ let $a,b,c$ be real numbers such that $au+bv+cv^{2}=0.$ Then $au=-bv-cv^{2}$ which implies that $a=0$ since $u$ is an idempotent. This results
$cv^{2}=-bv.$ If $b\neq 0,$ then we should have $c\neq 0$ which yields that $-(cb^{-1})v$ is an idempotent element what is an absurd. It follows that $b=c=0$ what is again an absurd.
\end{proof}

In the following subsections we build on ideas of \cite{Althoen}.

\subsection{Algebras that have a unique idempotent and don’t have solvable subalgebras generated by one element}
In this case, for any base $\{e_{1},e_{2}\}$ of $\mathds{R}^{2},$ where $e_{1}$ is a single-idempotent whose multiplication table is given by
Table II, we have $e_{2}^{2}\neq 0,$ by Lemma \ref{l51}. It follows that $a_{22}\neq 0,$ otherwise $\displaystyle (\frac{1}{b_{22}})e_{2}$ is another idempotent.

The proof for the result that follows is immediate and so we will omit it.

\begin{lemma}\label{l52} Let $\mathds{R}^{2}$ be a real nonassociative algebra whose multiplication table is given by Table II, where $e_{1}$ is a single-idempotent. Then there is a basis of $\mathds{R}^{2}$ $\{e_{1},f_{2}\}$ (where $f_{2}=\alpha e_{1}+\beta e_{2}$ with $\beta \neq 0$) such that $f_{2}^{2}=\pm e_{1}$ if and only if the system of equations
{\allowdisplaybreaks\begin{eqnarray}\allowdisplaybreaks
\left\{\begin{array}{ccccccc}\label{se1}
\alpha ^{2} & + & A\alpha \beta & + & a_{22}\beta ^{2} & = & \pm 1 \\
            &   & B\alpha \beta & + & b_{22}\beta ^{2} & = & 0
\end{array}\right.
\end{eqnarray}}
can be solved for at least one nonzero value of $\beta .$
\end{lemma}

\begin{lemma}\label{l53} Let $\mathds{R}^{2}$ be a real nonassociative algebra whose multiplication table is given by Table II, where $e_{1}$ is a single-idempotent and $b_{22}=0.$ Then there is a basis of $\mathds{R}^{2}$ $\{e_{1},f_{2}\}$ (where $f_{2}=\alpha e_{1}+\beta e_{2}$ with $\beta \neq 0$) such that $f_{2}^{2}=\pm e_{1}.$
\end{lemma}
\begin{proof} By Lemma \ref{l52}, the system of equations (\ref{se1}) can be solved taking $\alpha =0$ and selecting the sign on the right to match the sign of the coefficient of $\beta ^{2},$ in the first equation.
\end{proof}

\begin{lemma}\label{l54} Let $\mathds{R}^{2}$ be a real nonassociative algebra whose multiplication table is given by Table II, where $e_{1}$ is a single-idempotent and $b_{22}\neq 0.$ Then there is a basis of $\mathds{R}^{2}$ $\{e_{1},f_{2}\}$ (where $f_{2}=\alpha e_{1}+\beta e_{2}$ with $\beta \neq 0$) such that $f_{2}^{2}=\pm e_{1}$ if and only if $B\neq 0$ and $1-AB/b_{22}+B^{2}a_{22}/b_{22}^{2}\neq 0.$
\end{lemma}
\begin{proof} By Lemma \ref{l52}, the system of equations (\ref{se1}) is equivalent to the system of equations
{\allowdisplaybreaks\begin{eqnarray*}\allowdisplaybreaks
\left\{\begin{array}{cccccc}
(1-AB/b_{22}+B^{2}a_{22}/b_{22}^{2})\alpha ^{2} & & & & = & \pm 1 \\
                                                & B\alpha & + & b_{22}\beta & = & 0
\end{array}\right.
\end{eqnarray*}}
which can be solved by selecting the sign on the right to match the sign of the coefficient of $\alpha ^{2},$ in the first equation, and for at least one nonzero value of $\beta .$
\end{proof}

Every two-dimensional real nonassociative algebra satisfying the conditions of the Lemmas \ref{l53} or \ref{l54} admits the following multiplication tables:
{\allowdisplaybreaks\begin{eqnarray*}\allowdisplaybreaks
\begin{tabular}{c|cc}
\multicolumn{3}{c}{Table III} \\
   $\cdot $ & $e_{1}$ & $e_{2}$ \\ \hline
  $e_{1}$ & $e_{1}$ & $a_{12}e_{1}+b_{12}e_{2}$ \\
  $e_{2}$ & $a_{21}e_{1}+b_{21}e_{2}$ & $-e_{1}$
\end{tabular}\,\, \textrm{or}
\end{eqnarray*}}
or
{\allowdisplaybreaks\begin{eqnarray*}\allowdisplaybreaks
\begin{tabular}{c|cc}
\multicolumn{3}{c}{Table IV} \\
   $\cdot $ & $e_{1}$ & $e_{2}$ \\ \hline
  $e_{1}$ & $e_{1}$ & $a_{12}e_{1}+b_{12}e_{2}$ \\
  $e_{2}$ & $a_{21}e_{1}+b_{21}e_{2}$ & $e_{1}$
\end{tabular}.
\end{eqnarray*}}

\subsection{Algebras that have a unique idempotent and a solvable subalgebra generated by one element}
\begin{lemma}\label{l55} If $u$ is an idempotent and $v$ a nonzero element such that the subalgebra generated by $v$ is solvable, then $\{u,v\}$ is a basis of $\mathds{R}^{2}.$
\end{lemma}
\begin{proof} Let $a,b$ be real numbers such that $au+bv=0.$ Then $au=-bv$ which implies $a^{2}u=0,$ by Lemma \ref{l51}. This results $a=0$ which yields $b=0.$
\end{proof}

Every two-dimensional real nonassociative algebra satisfying the condition of the Lemma \ref{l55} admits a multiplication table of type
{\allowdisplaybreaks\begin{eqnarray*}\allowdisplaybreaks
\begin{tabular}{c|cc}
\multicolumn{3}{c}{Table V} \\
   $\cdot $ & $e_{1}$ & $e_{2}$ \\ \hline
  $e_{1}$ & $e_{1}$ & $a_{12}e_{1}+b_{12}e_{2}$ \\
  $e_{2}$ & $a_{21}e_{1}+b_{21}e_{2}$ & $0$
\end{tabular}.
\end{eqnarray*}}

\subsection{Algebras that have two different idempotents}
\begin{lemma}\label{l56} If $u$ and $v$ are two different idempotents, then $\{u,v\}$ is a basis of $\mathds{R}^{2}.$
\end{lemma}
\begin{proof} Let $a,b$ be real numbers such that $au+bv=0.$ Then $au=-bv$ which implies $a^{2}u=b^{2}v.$ Multiplying the penultimate identity by $a$ and subtracting it from the last we obtain $(a+b)b=0.$ This allows us to easily conclude that $a=b=0.$
\end{proof}

Every two-dimensional real nonassociative algebra satisfying the condition of the Lemma \ref{l56} admits a multiplication table of type
{\allowdisplaybreaks\begin{eqnarray*}\allowdisplaybreaks
\begin{tabular}{c|cc}
\multicolumn{3}{c}{Table VI} \\
   $\cdot $ & $e_{1}$ & $e_{2}$ \\ \hline
  $e_{1}$ & $e_{1}$ & $a_{12}e_{1}+b_{12}e_{2}$ \\
  $e_{2}$ & $a_{21}e_{1}+b_{21}e_{2}$ & $e_{2}$
\end{tabular}.
\end{eqnarray*}}

We end this part by observing that the complex number algebra $\mathds{C},$ the perplex number algebra $\mathds{P}$ and the dual number algebra $\mathds{D}$ are, up to isomorphism, two-dimensional real nonassociative algebras admitting the tables of type III, IV and V, respectively, all of them satisfying $A=0$ and $B=2.$

\begin{theorem}\label{thm51} Let $\mathds{R}^{2}$ be a real nonassociative algebra whose multiplication table is given by Table III. If $B\neq 0,$ then:
\begin{enumerate}
\item[(i)] for any element $u$ of $\mathds{R}^{2},$ $u^{2}=0$ if and only if $u=0;$
\item[(ii)] the norm $\|\cdot \|$ of $\mathds{R}^{2}$ satisfies the square inequality.
 \end{enumerate}
\end{theorem}
\begin{proof} (i) Let $u=ae_{1}+be_{2}$ be any element of $\mathds{R}^{2},$ where $a,b$ are real numbers. Then
{\allowdisplaybreaks\begin{eqnarray*}\allowdisplaybreaks
u^{2}&=&(ae_{1}+be_{2})(ae_{1}+be_{2})\\
&=&(a^{2}+(a_{12}+a_{21})ab-b^{2})e_{1}+(b_{12}+b_{21})abe_{2}\\
&=&(a^{2}+Aab-b^{2})e_{1}+Babe_{2}.
\end{eqnarray*}}
If $u^{2}=0,$ then $a^{2}+Aab-b^{2}=0$ and $Bab=0$ which implies $a=0$ or $b=0.$ In both cases, we easily conclude that $a=0$ and $b=0.$ Therefore $u^{2}=0$ implies $u=0.$\\
\noindent (ii) Let $u=ae_{1}+be_{2}$ be any nonzero element of $\mathds{R}^{2}.$ Two cases are considered. First case: $b\neq 0.$ Let $t$ be the real number such that $a=t b.$ Then
{\allowdisplaybreaks\begin{eqnarray*}\allowdisplaybreaks
\|u^{2}\|^{2}&=&(a^{2}+Aab-b^{2})^{2}+(Bab)^{2}\\
&=&\big((t ^{2}+At -1)^{2}+B^{2}t ^{2}\big)b^{4}
\end{eqnarray*}}
and
{\allowdisplaybreaks\begin{eqnarray*}\allowdisplaybreaks
\big(\|u\|^{2}\big)^{2}=(a^{2}+b^{2})^{2}=(t ^{2}+1)^{2}b^{4}.
\end{eqnarray*}}
The real function of one variable 
{\allowdisplaybreaks\begin{eqnarray*}\allowdisplaybreaks
 h(t)=\frac{(t^{2}+1)^{2}}{(t^{2}+At -1)^{2}+B^{2}t^{2}}
\end{eqnarray*}}
is defined and it is continuous on $\mathds{R}$ and $\displaystyle \lim _{t\rightarrow \pm \infty } h(t)=1.$ It follows that $h$ is a bounded function in $\mathds{R}.$ Thus, there is a positive real number $\eta _{1} $ such that $h(t)\leq \eta _{1}^{-2} ,$ for all real numbers $t.$ It follows that,
{\allowdisplaybreaks\begin{eqnarray*}\allowdisplaybreaks
&&\eta _{1} ^{-2}\|u^{2}\|^{2}-(\|u\|^{2})^{2}\\
&=&\eta _{1} ^{-2}\big((t ^{2}+At -1)^{2}+B^{2}t ^{2}\big)b^{4}-(t ^{2}+1)^{2}b^{4}\\
&=&\big(\eta _{1} ^{-2}\big((t ^{2}+At -1)^{2}+B^{2}t ^{2}\big)-(t ^{2}+1)^{2}\big)b^{4}\geq 0.
\end{eqnarray*}}
Therefore, $\|u^{2}\|\geq \eta _{1} \|u\|^{2}.$ Second case: $a\neq 0.$ From a similar reasoning used in the first case, we conclude that there is a positive real number $\eta _{2}$ such that $\|u^{2}\|\geq \eta _{2}\|u\|^{2}.$ Note that the values of $\eta _{1}$ and $\eta _{2}$ do not depend on the considered element $u.$ Taking $\eta =\min \{\eta _{1},\eta _{2}\},$ then we can conclude that $\|u^{2}\|\geq \eta \|u\|^{2},$ for all elements $u$ of $\mathds{R}^{2}.$
\end{proof}

The following example shows that it is not possible to characterize the orbits of the mapping $f_{c},$ according the Theorem \ref{thm31}, for a real nonassociative algebra whose multiplication table is given by Table III.

\begin{example}\label{e51} Let $\mathds{R}^{2}$ be a real nonassociative algebra whose multiplication table is given by Table III, where $B=0.$ Let us consider either the elements $\displaystyle c=\break \Big(\frac{-Aa+\sqrt{(A^{2}+4)a^{2}}}{2}\Big)e_{1}+ae_{2}$ and $\displaystyle u=\Big(\frac{-Ab+\sqrt{(A^{2}+4)b^{2}}}{2}\Big)e_{1}+be_{2}$ or $\displaystyle c=\break \Big(\frac{-Aa-\sqrt{(A^{2}+4)a^{2}}}{2}\Big)e_{1}+ae_{2}$ and $\displaystyle u=\Big(\frac{-Ab-\sqrt{(A^{2}+4)b^{2}}}{2}\Big)e_{1}+be_{2},$ where $a,b$ are any real numbers. We compute $u^{2}=0$ which implies that $f^{n}_{c}(u)=c,$ for every positive integer $n.$
\end{example}

\begin{theorem}\label{thm52} Let $\mathds{R}^{2}$ be a real nonassociative algebra whose multiplication table is given by Table IV. If $|A|<2$ or $B\neq 0,$ then:
\begin{enumerate}
\item[(i)] for any element $u$ of $\mathds{R}^{2},$ $u^{2}=0$ if and only if $u=0;$
\item[(ii)] the norm $\|\cdot \|$ of $\mathds{R}^{2}$ satisfies the square inequality.
 \end{enumerate}
\end{theorem}
\begin{proof} Let $u=ae_{1}+be_{2}$ be any element of $\mathds{R}^{2},$ where $a,b$ are real numbers. Then
{\allowdisplaybreaks\begin{eqnarray*}\allowdisplaybreaks
u^{2}&=&(ae_{1}+be_{2})(ae_{1}+be_{2})\\
&=&(a^{2}+(a_{12}+a_{21})ab+b^{2})e_{1}+(b_{12}+b_{21})abe_{2}\\
&=&(a^{2}+Aab+b^{2})e_{1}+Babe_{2}.
\end{eqnarray*}}
If $u^{2}=0,$ then $a^{2}+Aab+b^{2}=0$ and $Bab=0.$ Two cases are considered. First case: $|A|<2.$ Then $a^{2}+Aab+b^{2}=0$ implies $\displaystyle \big(a+\frac{A}{2}b\big)^{2}=\big(\frac{A^{2}-4}{4}\big)b^{2}.$ This shows that $b^{2}=0$ which yields $a=0$ and $b=0.$ Second case: $B\neq 0.$ Then $a=0$ or $b=0$ which leads to $a=0$ and $b=0.$ This shows that $u^{2}=0$ implies $u=0.$\\
\noindent (ii) Let $u=ae_{1}+be_{2}$ be any nonzero element of $\mathds{R}^{2}.$ Two cases are considered. First case: $b\neq 0.$ Let $t$ be a real number such that $a=tb.$ Then
{\allowdisplaybreaks\begin{eqnarray*}\allowdisplaybreaks
\|u^{2}\|^{2}&=&(a^{2}+Aab+b^{2})^{2}+(Bab)^{2}\\
&=&\big((t ^{2}+At +1)^{2}+B^{2}t ^{2}\big)b^{4}
\end{eqnarray*}}
and
{\allowdisplaybreaks\begin{eqnarray*}\allowdisplaybreaks
\big(\|u\|^{2}\big)^{2}=(a^{2}+b^{2})^{2}=(t ^{2}+1)^{2}b^{4}.
\end{eqnarray*}}
The real function of one variable 
{\allowdisplaybreaks\begin{eqnarray*}\allowdisplaybreaks
h(t)=\frac{(t^{2}+1)^{2}}{(t^{2}+At +1)^{2}+B^{2}t^{2}}
\end{eqnarray*}}
is defined and it is continuous on $\mathds{R}$ and $\displaystyle \lim _{t\rightarrow \pm \infty } h(t)=1.$ It follows that $h$ is a bounded function in $\mathds{R}.$ Thus, there is a positive real number $\eta _{1} $ such that $h(t)\leq \eta _{1}^{-2} ,$ for all real numbers $t.$ Thus,
{\allowdisplaybreaks\begin{eqnarray*}\allowdisplaybreaks
&&\eta _{1} ^{-2}\|u^{2}\|^{2}-(\|u\|^{2})^{2}\\
&=&\eta _{1} ^{-2}\big((t ^{2}+At +1)^{2}+B^{2}t ^{2}\big)b^{4}-(t ^{2}+1)^{2}b^{4}\\
&=&\big(\eta _{1} ^{-2}\big((t ^{2}+At -1)^{2}+B^{2}t ^{2}\big)-(t ^{2}+1)^{2}\big)b^{4}\geq 0.
\end{eqnarray*}}
Therefore, $\|u^{2}\|\geq \eta _{1}\|u\|^{2}.$ Second case: $a\neq 0.$ From a similar reasoning used in the previous case, we conclude that there is a positive real number $\eta _{2}$ such that $\|u^{2}\|\geq \eta _{2}\|u\|^{2}.$ Note that the values of $\eta _{1}$ and $\eta _{2}$ do not depend on the considered element $u.$ Taking $\eta =\min \{\eta _{1},\eta _{2}\},$ then we obtain $\|u^{2}\|\geq \eta \|u\|^{2},$ for all elements $u$ of $\mathds{R}^{2}.$
\end{proof}

The example below presents a real nonassociative algebra whose multiplication table is given by Table IV, where it is not possible to characterize the orbits of the mapping $f_{c},$ according the Theorem \ref{thm31}.

\begin{example}\label{e52} Let $\mathds{R}^{2}$ be a real nonassociative algebra whose multiplication table is given by Table IV, where $|A|\geq 2$ and $B=0.$ Let us consider either the elements $\displaystyle c=\Big(\frac{-Aa+ \sqrt{(A^{2}-4)a^{2}}}{2}\Big)e_{1}+ae_{2}$ and $\displaystyle u=\Big(\frac{-Ab+ \sqrt{(A^{2}-4)b^{2}}}{2}\Big)e_{1}+be_{2}$ or $\displaystyle c=\Big(\frac{-Aa- \sqrt{(A^{2}-4)a^{2}}}{2}\Big)e_{1}+ae_{2}$ and $\displaystyle u=\Big(\frac{-Ab- \sqrt{(A^{2}-4)b^{2}}}{2}\Big)e_{1}+be_{2},$ where $a,b$ are any real numbers. We compute $u^{2}=0$ which implies that $f^{n}_{c}(u)=c,$ for every positive integer $n.$
\end{example}

The following example presents a real nonassociative algebra whose multiplication table is given by Table IV where there are empty filled Julia sets.

\begin{example}\label{e53} Let $\mathds{R}^{2}$ be a real nonassociative algebra whose multiplication table is given by Table IV, where $|B|\geq |A|=0,$ and an element $c=\alpha e_{1}+\beta e_{2}$ such that $\alpha >1.$ Then for any element $u=ae_{1}+be_{2}$ of $\mathds{R}^{2},$ where $a,b$ are real numbers, we have
{\allowdisplaybreaks\begin{eqnarray*}\allowdisplaybreaks
f_{c}(u)&=&u^{2}+c\\
&=&(ae_{1}+be_{2})(ae_{1}+be_{2})+(\alpha e_{1}+\beta e_{2})\\
&=&(a^{2}+Aab+b^{2}+\alpha )e_{1}+(Bab+\beta )e_{2}\\
&=&(a^{2}+b^{2}+\alpha )e_{1}+(Bab+\beta )e_{2}\\
&=&(\|u\|^{2}+\alpha )e_{1}+(Bab+\beta )e_{2}
\end{eqnarray*}}
which results $\|f_{c}(u)\|\geq \|u\|^{2}+\alpha .$ Assume that for some positive integer $l$ holds $\|f^{l}_{c}(u)\|\geq (\|u\|^{2}+\alpha )^{l},$ for all elements $u$ of $\mathds{R}^{2}.$ Then for the integer $l+1$ we have
{\allowdisplaybreaks\begin{eqnarray*}\allowdisplaybreaks
\|f^{l+1}_{c}(u)\|&=&\|f_{c}(f^{l}_{c}(u))\|\geq \|f^{l}_{c}(u)\|^{2}+\alpha \\
&\geq & \big((\|u\|^{2}+\alpha )^{l}\big)^{2}+\alpha \geq (\|u\|^{2}+\alpha )^{l+1},
\end{eqnarray*}}
for all elements $u$ of $\mathds{R}^{2}.$ Thus, by the principle of mathematical induction we prove that $\|f^{n}_{c}(u)\|\geq (\|u\|^{2}+\alpha )^{n},$ for all elements $u$ of $\mathds{R}^{2}$ and every positive integer $n.$ This allows us to conclude that the orbit of any element $u$ of $\mathds{R}^{2}$ is unbounded. This shows that $\mathcal{K}_{\mathds{R}^{2}}(f_{c})=\emptyset .$
\end{example}

The theorem and example that follow are related to the class of algebras whose multiplication table is given by Table II, where $b_{22}\neq 0,$ but they do not satisfy the conditions of the Lemma \ref{l54}.

\begin{theorem}\label{thm53} Let $\mathds{R}^{2}$ be a real nonassociative algebra whose multiplication table is given by Table II, where $b_{22}\neq 0.$ If $B=0,$ then:
\begin{enumerate}
\item[(i)] for any element $u$ of $\mathds{R}^{2},$ $u^{2}=0$ if and only if $u=0;$
\item[(ii)] the norm $\|\cdot \|$ of $\mathds{R}^{2}$ satisfies the square inequality.
 \end{enumerate}
\end{theorem}
\begin{proof} (i) Let $u=ae_{1}+be_{2}$ be an element of $\mathds{R}^{2},$ where $a,b$ are real numbers. Then
{\allowdisplaybreaks\begin{eqnarray*}\allowdisplaybreaks
u^{2}&=&(ae_{1}+be_{2})(ae_{1}+be_{2})\\
&=&(a^{2}+(a_{12}+a_{21})ab+a_{22}b^{2})e_{1}\\
&&+((b_{12}+b_{21})ab+b_{22}b^{2})e_{2}\\
&=&(a^{2}+Aab+a_{22}b^{2})e_{1}+(Bab+b_{22}b^{2})e_{2}\\
&=&(a^{2}+Aab+a_{22}b^{2})e_{1}+(b_{22}b^{2})e_{2}.
\end{eqnarray*}}
If $u^{2}=0,$ then $a^{2}+Aab+a_{22}b^{2}=0$ and $b_{22}b^{2}=0.$ It follows that, $b^{2}=0$ which results $a=0$ and $b=0.$ This shows that $u^{2}=0$ implies $u=0.$\\
\noindent (ii) Let $u=ae_{1}+be_{2}$ be a nonzero element of $\mathds{R}^{2},$ where $a,b$ are real numbers. Two cases are considered. First case: $b\neq 0.$ Let $t$ be the real number such that $a=tb.$ Then
{\allowdisplaybreaks\begin{eqnarray*}\allowdisplaybreaks
\|u^{2}\|^{2}&=&(a^{2}+Aab+a_{22}b^{2})^{2}+(b_{22}b^{2})^{2}.\\
&=&(t ^{2}+At +a_{22})^{2}b^{4}+b_{22}^{2}b^{4}\\
&=&\big((t ^{2}+At +a_{22})^{2}+b_{22}^{2}\big)b^{4}
\end{eqnarray*}}
and
{\allowdisplaybreaks\begin{eqnarray*}\allowdisplaybreaks
(\|u\|^{2})^{2}=\big(a^{2}+b^{2}\big)^{2}=(t ^{2}+1)^{2}b^{4}.
\end{eqnarray*}}
The real function of one variable 
{\allowdisplaybreaks\begin{eqnarray*}\allowdisplaybreaks
h(t)=\frac{(t^{2}+1)^{2}}{(t^{2}+At+a_{22})^{2}+b_{22}^{2}}
\end{eqnarray*}}
is defined and it is continuous on $\mathds{R}$ and $\displaystyle \lim _{t\rightarrow \pm \infty } h(t)=1.$ This shows that $h$ is a bounded function in $\mathds{R}.$ Thus, there is a positive real number $\eta _{1} $ such that $h(t)\leq \eta _{1}^{-2} ,$ for all real numbers $t.$ Thus,
{\allowdisplaybreaks\begin{eqnarray*}\allowdisplaybreaks
&&\eta _{1} ^{-2}\|u^{2}\|^{2}-(\|u\|^{2})^{2}\\
&=&\eta _{1} ^{-2}\big((t ^{2}+At +a_{22})^{2}+b_{22}^{2}\big)b^{4}-(t ^{2}+1)^{2}b^{4}\\
&=&\big(\eta _{1} ^{-2}\big((t ^{2}+At +a_{22})^{2}+b_{22}^{2}\big)-(t ^{2}+1)^{2}\big)b^{4}\geq 0.
\end{eqnarray*}}
Therefore, $\|u^{2}\|\geq \eta _{1}\|u\|^{2}.$ Second case: $a\neq 0.$ From a similar reasoning used in the previous case, we conclude that there is a positive real number $\eta _{2} $ such that $\|u^{2}\|\geq \eta _{2}\|u\|^{2}.$ Note that the values of $\eta _{1}$ and $\eta _{2}$ do not depend on the considered element $u.$ Taking $\eta =\min \{\eta _{1}, \eta _{2}\},$ then we can conclude that $\|u^{2}\|\geq \eta \|u\|^{2},$ for all elements $u$ of $\mathds{R}^{2}.$
\end{proof}

\begin{example}\label{e54} Let $\mathds{R}^{2}$ be a real nonassociative algebra whose multiplication table is given by Table II, where $A=B=2,$ $\displaystyle a_{22}=\frac{3}{4}$ and $b_{22}=1.$ Then $1-AB/b_{22}+B^{2}a_{22}/b_{22}^{2}=0.$ Let us consider the elements $\displaystyle c=ae_{1}+(-2 a)e_{2}$ and $\displaystyle u=be_{1}+(-2b)e_{2},$ where $a,b$ are any real numbers. We compute $u^{2}=0$ which implies that $f^{n}_{c}(u)=c,$ for every positive integer $n.$
\end{example}

The following example shows that it is not possible to characterize the orbits of the mapping $f_{c},$ according the Theorem \ref{thm31}, for real nonassociative algebras whose multiplication table is given by Table V.

\begin{example}\label{e55} Let $\mathds{R}^{2}$ be a real nonassociative algebra whose multiplication table is given by Table V, where $A$ and $B$ are any real numbers. Let us consider the elements $c=0e_{1}+ae_{2}$ and $u=0e_{1}+be_{2},$ where $a,b$ are any real numbers. We compute $u^{2}=0$ which implies that $f^{n}_{c}(u)=c,$ for every positive integer $n.$
\end{example}

\begin{theorem}\label{thm54} Let $\mathds{R}^{2}$ be a real nonassociative algebra whose multiplication table is given by Table VI. If $AB\neq 1,$ then:
\begin{enumerate}
\item[(i)] for any element $u$ of $\mathds{R}^{2},$ $u^{2}=0$ if and only if $u=0;$
\item[(ii)] the norm $\|\cdot \|$ of $\mathds{R}^{2}$ satisfies the square inequality.
 \end{enumerate}
\end{theorem}
\begin{proof} (i) Let $u=ae_{1}+be_{2}$ be an element of $\mathds{R}^{2},$ where $a,b$ are real numbers. Then
{\allowdisplaybreaks\begin{eqnarray*}\allowdisplaybreaks
u^{2}&=&(ae_{1}+be_{2})(ae_{1}+be_{2})\\
&=&(a^{2}+(a_{12}+a_{21})ab)e_{1}+((b_{12}+b_{21})ab+b^{2})e_{2}\\
&=&(a^{2}+Aab)e_{1}+(Bab+b^{2})e_{2}.
\end{eqnarray*}}
If $u^{2}=0,$ then $a^{2}+Aab=0$ and $Bab+b^{2}=0.$ It follows that if $a=0,$ then $b^{2}=0$ which results $b=0.$ Now, if $a+Ab=0,$ then $(-AB+1)b^{2}=0$ which yields that $b^{2}=0.$ This results $a=0$ and $b=0.$ This allows us to conclude that $u^{2}=0$ implies $u=0.$\\
\noindent (ii) Let $u=ae_{1}+be_{2}$ be a nonzero element of $\mathds{R}^{2},$ where $a,b$ are real numbers. Two cases are considered. First case: $b\neq 0.$ Let $t$ be the real number such that $a=tb.$ Then
{\allowdisplaybreaks\begin{eqnarray*}\allowdisplaybreaks
\|u^{2}\|^{2}&=&(a^{2}+Aab)^{2}+(Bab+b^{2})^{2}\\
&=&(t ^{2}+At )^{2}b^{4}+(Bt +1)^{2}b^{4}\\
&=&\big((t ^{2}+At )^{2}+(Bt +1)^{2}\big)b^{4}
\end{eqnarray*}}
and
{\allowdisplaybreaks\begin{eqnarray*}\allowdisplaybreaks
(\|u\|^{2})^{2}=\big(a^{2}+b^{2}\big)^{2}=(t ^{2}+1)^{2}b^{4}.
\end{eqnarray*}}
The real function of one variable 
{\allowdisplaybreaks\begin{eqnarray*}\allowdisplaybreaks
h(t)=\frac{(t^{2}+1)^{2}}{(t^{2}+At)^{2}+(Bt+1)^{2}}
\end{eqnarray*}}
is defined and it is continuous on $\mathds{R}$ and $\displaystyle \lim _{t\rightarrow \pm \infty } h(t)=1.$ This shows that $h$ is a bounded function in $\mathds{R}.$ Thus, there is a positive real number $\eta _{1} $ such that $h(t)\leq \eta _{1}^{-2},$ for all real numbers $t.$ It follows that,
{\allowdisplaybreaks\begin{eqnarray*}\allowdisplaybreaks
&&\eta _{1} ^{-2}\|u^{2}\|^{2}-(\|u\|^{2})^{2}\\
&=&\eta _{1} ^{-2}\big((t ^{2}+At )^{2}+(Bt +1)^{2}\big)b^{4}-(t ^{2}+1)^{2}b^{4}\\
&=&\big(\eta _{1} ^{-2}\big((t ^{2}+At )^{2}+(Bt +1)^{2}\big)-(t ^{2}+1)^{2}\big)b^{4}\geq 0.
\end{eqnarray*}}
Therefore $\|u^{2}\|\geq \eta _{1}\|u\|^{2}.$ Second case: $a\neq 0.$ From a similar reasoning used in the previous case, we conclude that there is a positive real number $\eta _{2} $ such that $\|u^{2}\|\geq \eta _{2}\|u\|^{2}.$ Note that the values of $\eta _{1}$ and $\eta _{2}$ do not depend on the considered element $u.$ Taking $\eta =\min \{\eta _{1}, \eta _{2}\},$ then we conclude that $\|u^{2}\|\geq \eta \|u\|^{2},$ for all elements $u$ of $\mathds{R}^{2}.$
\end{proof}

The example below presents a real nonassociative algebra whose multiplication table is given by Table VI where it is not possible to characterize the orbits of the mapping $f_{c},$ according the Theorem \ref{thm31}.

\begin{example} Let $\mathds{R}^{2}$ be a real nonassociative algebra whose multiplication table is given by Table VI, where $A$ and $B$ are any nonzero real numbers such that $AB=1.$ Let us consider the elements $c=ae_{1}-Bae_{2}$ and $u=be_{1}-Bbe_{2},$ where $a,b$ are any real numbers. We compute $u^{2}=0$ which implies that $f^{n}_{c}(u)=c,$ for every positive integer $n.$
\end{example}

The following example presents a real nonassociative algebra whose multiplication table is given by Table VI and in which there are empty filled Julia sets.

\begin{example} Let $\mathds{R}^{2}$ be a real nonassociative algebra whose multiplication table is given by Table VI, where $A=B=0,$ and an element $c=\alpha e_{1}+\beta e_{2}$ such that $\alpha >1$ and $\beta \geq \frac{1}{2}.$ Then for any element $u=ae_{1}+be_{2}$ of $\mathds{R}^{2},$ where $a,b$ are real numbers, we have
{\allowdisplaybreaks\begin{eqnarray*}\allowdisplaybreaks
\|f_{c}(u)\|&=&\|u^{2}+c\|=\|(ae_{1}+be_{2})(ae_{1}+be_{2})+(\alpha e_{1}+\beta e_{2})\|\\
&=&\|(a^{2}+\alpha )e_{1}+(b^{2}+\beta )e_{2}\|=\sqrt{(a^{2}+\alpha )^{2}+(b^{2}+\beta )^{2}}\\
&\geq &\sqrt{2a^{2}\alpha +2 b^{2}\beta+\alpha ^{2}+\beta ^{2}}\geq \sqrt{\|u\|^{2}+\|c\|^{2}}
\end{eqnarray*}}
which results $\|f_{c}(u)\|\geq \sqrt{\|u\|^{2}+\|c\|^{2}}.$ Assume that for some positive integer $l$ holds $\|f^{l}_{c}(u)\|\geq \sqrt{\|u\|^{2}+l\|c\|^{2}},$ for all elements $u$ of $\mathds{R}^{2}.$ Then for the integer $l+1$ we have
{\allowdisplaybreaks\begin{eqnarray*}\allowdisplaybreaks
\|f^{l+1}_{c}(u)\|&=&\|f_{c}(f^{l}_{c}(u))\|\geq \sqrt{\|f^{l}_{c}(u)\|^{2}+\|c\|^{2}}\\
&\geq & \sqrt{\|u\|^{2}+l\|c\|^{2}+\|c\|^{2}}=\sqrt{\|u\|^{2}+l\|c\|^{2}+\|c\|^{2}}\\
&= &\sqrt{\|u\|^{2}+(l+1)\|c\|^{2}}
\end{eqnarray*}}
for all elements $u$ of $\mathds{R}^{2}.$ Thus, by the principle of mathematical induction we prove that $\|f^{n}_{c}(u)\|\geq \sqrt{\|u\|^{2}+n\|c\|^{2}},$ for all elements $u$ of $\mathds{R}^{2}$ and every positive integer $n.$ This allows us to conclude that the orbit of any element $u$ of $\mathds{R}^{2}$ is unbounded. This shows that $\mathcal{K}_{\mathds{R}^{2}}(f_{c})=\emptyset .$
\end{example}

We conclude this paper with a result that is a direct consequence of the Theorem \ref{thm33}.

\begin{corollary} Let $\mathds{R}^{2}$ be a real nonassociative algebra whose multiplication table is given either by Table II, where $b_{22}\neq 0$ and $B=0,$ by Table III, where $B\neq 0,$ by Table IV, where $|A|<2$ or $B\neq 0,$ or by Table VI, where $AB\neq 1.$ Then:
\begin{enumerate}
\item[(a)] for every element $c$ of $\mathds{R}^{2}$ the filled Julia sets $\mathcal{K}_{\mathds{R}^{2}}(f_{c})$ associated with $f_{c}$ are compact sets of $\mathds{R}^{2};$
\item[(b)] the Mandelbrot sets $\mathcal{M}_{\mathds{R}^{2}}$ are compact sets of $\mathds{R}^{2}.$
\end{enumerate}
\end{corollary}

%

\end{document}